\documentclass[11pt,a4paper]{amsart}
\newcommand{\correctspacing}{\singlespacing}

\usepackage{amsmath,amsthm,amsfonts,amssymb,epsfig}
\usepackage{setspace}
\usepackage{xcolor}
\correctspacing
\usepackage{graphicx}

\usepackage{amssymb}
\usepackage{graphicx}
\usepackage{color}
 
\usepackage{amssymb}
\usepackage{amsmath}
\usepackage{amsthm}
\usepackage{amscd}
\usepackage{url}

\usepackage{enumerate}

\theoremstyle{plain}
\newtheorem{theorem}{Theorem}[section]
\newtheorem{pro}[theorem]{Proposition}

\newtheorem{lemma}[theorem]{Lemma}
\newtheorem{que}[theorem]{Question}
\newtheorem{assumption}[theorem]{Assumption}

\theoremstyle{definition}
\newtheorem{definition}[theorem]{Definition}

\theoremstyle{definition}

\newtheorem{example}[theorem]{Example}

\DeclareMathOperator{\supp}{supp}

\renewcommand{\P}{\gamma}%{\mathbb{P}}

\newcommand{\R}{\mathbb{R}}

\newcommand{\F}{\mathcal{F}}
\renewcommand{\H}{\mathcal{H}}

\usepackage{bbm}

\renewcommand{\phi}{\varphi}

\newcommand{\Fixme}[1]{\marginpar{}}%{\marginpar{\singlespacing \tiny #1}}
\newcommand{\gray}[1]{\marginpar{}} %{\marginpar{{#1}}}

\renewcommand{\subset}{\subseteq}

\newcommand{\OB}{Ob{\l}{\'o}j}

 \providecommand{\cM}{\mathcal{M}}
\providecommand{\P}{\mathbb{P}}

\providecommand{\NN}{\mathbb{N}}  \providecommand{\RR}{\mathbb{R}}
\providecommand{\QQ}{\mathbb{Q}} 
\providecommand{\cP}{\mathcal{P}}

\providecommand{\cM}{\mathcal{M}} 
 
 \providecommand{\cM}{\mathcal{M}} 
  \providecommand{\cF}{\mathcal{F}}
\providecommand{\cR}{\mathcal{R}}
\providecommand{\indy}{\mathbf{1}}

\newcommand{\be}{\begin{equation}} \newcommand{\beo}{\begin{equation*}}
\newcommand{\ee}{\end{equation}}    \newcommand{\eeo}{\end{equation*}}
\newcommand{\bea}{\begin{eqnarray}}
\newcommand{\eea}{\end{eqnarray}}
\newcommand{\beao}{\begin{eqnarray*}} \newcommand{\eeao}{\end{eqnarray*}}

\renewcommand{\cM}{\Pi}

\newcommand{\secret}[1]{}

\author{Mathias Beiglb\"ock}
\author{Claus Griessler}
\thanks{We acknowledge financial support through FWF-projects  P26736 and Y00782. We also thank Manu Eder for many helpful comments. }
\date{\today}
\title[A  monotonicity principle]%{A cornucopia of cyclical monotonicity}
{A land of  monotone plenty}

\begin{document}
\maketitle
\begin{abstract}
A fundamental concept in optimal transport is $c$-cyclical monotonicity: it allows to link the optimality of transport plans to the geometry of their support sets. Recently, related concepts have been successfully applied in the multi-marginal version of the transport problem as well as in the martingale transport problem which arises from model-independent finance. 

We establish a unifying concept of \emph{$c$-monotonicity / finitistic optimality} which describes the geometric structure  of optimizers of a generalized moment problem. This allows us to strengthen known results in optimal martingale transport and for a transport problem with a continuum of marginals.

\smallskip

If the optimization problem can be formulated as a multi-marginal transport problem, potentially with additional linear constraints, our contribution is parallel to a recent result of Zaev.

\smallskip

\end{abstract}

\section{Introduction}

 \subsection{Motivation from optimal transport}
Given probabilities $\mu$ and $\nu$ on Polish 
spaces $X$ and  $Y$, and a cost function $c:X\times Y\to \R_+$, 
the Monge-Kan\-to\-ro\-vich problem
is to find a cost minimizing transport plan. More precisely, 
% cf.\ \cite{Vi03,Vi09}.
writing $\Pi(\mu,\nu)$  for the set of all measures on
$X\times Y$ with $X$-marginal $\mu$ and $Y$-marginal $\nu$, 
% Associated to a   are the transport costs
% $ \int c\, d\P.$ 
% The 
% Monge-Kantorovich problem is to determine the value
the problem is to find 
\begin{equation}\label{G1}\tag{OT}
\inf \left\{ \int c\, d\P:\P\in \Pi(\mu,\nu) \right\} \, 
\end{equation}
and to identify an optimal transport plan $\P^* \in \Pi(\mu,\nu)$.

The concept of  \emph{$c$-cyclical monotonicity} leads to a \emph{geometric characterization} of optimal couplings. %Its relevance was fully recognized by Gangbo and McCann \cite{GaMc96}. 
Its relevance for \eqref{G1} has been fully
recognized by Gangbo and McCann \cite{GaMc96}, based on earlier work of Knott
and Smith \cite{KnSm92} and R\"uschendorf \cite{Ru96} among others.

%Heuristically,  a transport plan is $c$-cyclically monotone if it cannot be improved by means of cyclical rerouting. More precisely,  a  set $\Gamma \subseteq X \times Y$  is $c$-cyclically monotone, if for any $(x_1, y_1), \dots, (x_n, y_n) \in \Gamma$, one has, with $y_{n+1}=y_1$, 

%\be\label{cm}
%\sum_{i=1}^n c(x_i, y_i) \leq \sum_{i=1}^n c(x_i, y_{i+1}), 
%\ee
%and a transport plan is  $c$-cyclically monotone if it is concentrated on such set.

A set $\Gamma \subseteq X \times Y$ is $c$-cyclically monotone if any measure $\alpha$, that is finite and supported on finitely many points in $\Gamma$, is a cost-minimizing transport between its marginals. I.e., if $\alpha'$ has the same marginals as $\alpha$, then 
\be
\int c \, d \alpha \leq \int c \, d \alpha'.
\ee
A transport plan $\gamma$ is called $c$-cyclically monotone if it is concentrated on such a set $\Gamma$, i.e.\ if there is such a $\Gamma$ with $\gamma(\Gamma) = 1$.\footnote{The more familiar way of stating $c$-cyclical monotonicity for a set $\Gamma$ is to assert  that  for any  $(x_{1}, y_{1}), \dots, (x_{n}, y_{n}) \in \Gamma$,  with the convention $y_{n+1}=y_1$, 
\beo
\sum_{i=1}^n c(x_{i}, y_{i}) \leq \sum_{i=1}^n c(x_{i}, y_{i+1}).
\eeo
We have used the equivalent formulation above as it is not inherently two-dimensional and serves our exposition more directly. For the equivalence 
cf.\ \cite[Exercise 2.21]{Vi03}.}

Connecting optimality and $c$-cyclical monotonicity is technically intricate. A series of contributions 
(\cite{AmPr03, Pr07, ScTe09,BeGoMaSc09, BiCa10} among others) led to the following clear-cut characterization: 
\begin{theorem}[Monotonicity Principle]\label{LGTrans}
Let $c:X\times Y \to [0,\infty)$ be Borel measurable and assume that $\P\in \Pi(\mu, \nu)$ is a transport plan with finite costs $\int c\, d\P \in \R_+$. Then $\P$ is optimal if and only if $\P$ is $c$-cyclically monotone.
\end{theorem}
The importance of this result stems from the observation that it is often an elementary  and feasible task to see whether a transport behaves optimally on a  finite number of points. But  this would be a priori of no help for a problem where single points do not carry positive mass.  Theorem \ref{LGTrans}  provides the required remedy to this obstacle as it establishes  the connection to optimality on a ``pointwise'' level.

 \subsection{Recent developments and aims of this article}

Recently several variants of \eqref{G1} have been discussed in the literature: the multi-marginal transport problem (see \cite{Ke84, Car03, Pa11fm,Pa12fm, KiPa13}), the martingale transport problem (\cite{BeHePe13, GaHeTo14, DoSo14a, DoSo14b, BeJu16, BeCoHu16, HeTo13, BeNuTo16,  NuSt16, GhKiLi16} among others), and problems where a continuum of marginals is prescribed, \cite{Pa13,Pa13b}.  Having cyclical monotonicity in mind, the problem in  \cite{Pa13} seems of particular interest: here, Pass presents a solution that is of Monge-type and appears very natural, yet the proof of its optimality and uniqueness appears rather technical and relies on assumptions that might  be difficult to verify in practice. But it is apparent that this solution is the only  transport that fulfills an infinite-dimensional analogue of cyclical monotonicity. We were thus drawn to the question  whether a suitable notion of $c$-cyclical monotonicity could  prove useful for such extended problems  by reducing the technical level and leading to stronger results.

The main goal of this article is therefore to establish  a \emph{monotonicity principle}  as an analogue of the ``necessary"-part of Theorem \ref{LGTrans} in a rather wide generality.  %By doing so, we get a unified proof forthat can be applied to both the transport problem and its several variants that have been discussed recently: the multi-marginal transport problem (see \cite{Ke84, Car03, Pa11fm,Pa12fm, KiPa13}), the martingale transport problem (\cite{BeHePe12, GaHeTo13, BeJu12, HeTo13,  DoSo13, DoSo13b} among others), and a continuum-marginal problem, \cite{Pa13,Pa13b}. %Regarding the latter, our approach can indeed strengthen known results, as the intuitive nature of the monotonicity principle makes do with considerably less regularity assumptions.
More precisely, we use the framework of a generalized moment problem \eqref{Primal}, define a general notion of $c$-monotonicity and then establish that optimizers are $c$-monotone. To this end, we build on ideas from \cite{BeJu16}, where, mimicking the idea of $c$-cyclical monotonicity,  a notion of ``finitistic optimality'' was already introduced for the martingale transport problem and optimizers were shown to fulfill that criterion.\footnote{In fact, in many instances of the martingale transport problem, finitistic optimality is also sufficient for optimality.}   

Our Theorem \ref{MainTheorem} allows to obtain improved versions of  the results from  \cite{BeJu16} and  \cite[Proposition 2.3]{KiPa13},  and it includes  one implicaton of the classical result stated in Theorem \ref{LGTrans}.  Finally, we use the result to prove a strengthened version of  Pass' Monge-type result.  In contrast to Pass' original derivation we do not require additional assumptions on the payoff functional or the prescribed marginals which might be difficult to verify in his intended applications.  %This might be due to the more intuitive structure of the monotonicity approach, that seems to fit well situations where there is a hint for a "natural" solution.

We note that, although  \eqref{Primal} constitutes a classical problem in probability,  and  it is  well known that \eqref{G1} and its variants  fit into this framework (see e.g.\   \cite{Ke68} and \cite{La10}), the general optimality criterion of $c$-monotonicity we state in Theorem \ref{MainTheorem} is new to the best of our knowledge.

%Given the importance of $c$-cyclical monotonicity it is natural to look for a related concept applicable in these variants of the transport problem. Kim and Pass \cite{KiPa13} introduced a notion of $c$-monotonicity, necessary for optimality in the context of the multi-marginal transport problem (\cite[Proposition 2.3]{KiPa13}).  They use it  to develop a general condition on the cost function which is sufficient to imply existence of a Monge solution and uniqueness results in the multi-marginal optimal transport problem.  

We point out a particular novelty of the approach in this article: in all the instances where the monotonicity principle was previously known, the minimization problem \eqref{Primal} admits a well understood dual problem and it is known that there is no duality gap. In the literature  on the Monge-Kantorovich problem,  it is well known that the absence of a duality gap can be used to show that optimal transport plans are cyclically monotone, see e.g.\ \cite[Exercise 2.38]{Vi03}. 
In fact, assuming certain regularity assumptions, this argument could be used to establish Theorem \ref{MainTheorem} whenever there is no duality gap. The advantage of the approach presented below is twofold. On the one hand it allows to derive the desired implication virtually without regularity assumptions. More importantly, it is applicable also in situations where duality is either unknown or known to fail (cf.\ \cite[Section 3.4]{AnNa87} for such cases). 

% Below we present an application of Theorem \ref{MainTheorem} where this is crucial: 
%This argument carries over to more general setups, in particular  weak duality implies 

%Theorem \ref{MainTheorem} allows to obtain improved versions of  the results from  \cite{BeJu12} and  \cite[Proposition 2.3]{KiPa13},  and it includes  one implicaton of the classical result stated in Theorem \ref{LGTrans}. To exemplify the result's applicability beyond optimization on finite products of spaces, we prove a strengthened version of  Pass' Monge-type result for a continuum of marginals \cite{Pa13}. In contrast to Pass' original result we do not require additional assumptions on the payoff functional or the prescribed marginals.  This might be due to the more intuitive structure of the monotonicity approach, that seems to fit well situations where there is a hint for a "natural" solution.
 % Pass uses his result to derive an infinite dimensional rearrangement inequality and upper bounds on solutions to parabolic PDE among other applications. %This result can be used to provide model independent bounds on the prices of certain derivatives in mathematical finance. As an another application, we obtain refined phase-space bounds in quantum physics. %Pass gives a number of applications for his result.
 %\marginpar{Text has changed} 

We conclude this section by a precise statement of the problem, the definition of $c$-monotonicity and the optimality criterion of Theorem \ref{MainTheorem}. For readability we postpone its proof to the last section, Section 4. Section 2 shows how some problems can be written in our framework. % and shows how our concept of $c$-monotonicity emerges naturally from cyclical monotonicity. 
We also give a counterexample on the ``sufficiency''-part of Theorem \ref{LGTrans} in the general situation. Section 3  deals with  the problem from \cite{Pa13} in  light of Theorem \ref{MainTheorem}.

\subsection{The basic optimization problem}
Let  $E$ be a  Polish space and $c: E \rightarrow \R $ a Borel measurable cost function. %The examples encountered later will have $E= \bigl( \RR^d \bigr)^n$ and $E = C[0,T]$, the space of continuous functions $[0,T]\rightarrow \RR$ with the topology of uniform convergence.

We fix a set  $\cF$ of Borel-measurable functions on  $E$ and  write  $\cM_{\cF}$ for the set of  probability measures $\P$ on $E$ for which $\int f \; d\P = 0 $ for all $f\in \cF$.\footnote{
By asserting that $\int f\, d\P= 0$ we implicitly understand that this integral exists.
Throughout this article we use the convention $+\infty-\infty=+\infty$.
} The  generalized moment problem is then to minimize    the total cost choosing from $\cM_{\cF}$, i.e.
\begin{align}\label{Primal}\tag{GMP}
\min_{ \P\in \cM_{\cF}} \int c \; d\P.
\end{align}

\subsection{A general concept of $c$-monotonity and main result}
\secret{
A transport plan $\gamma$ is called $c$-cyclically monotone when it is concentrated on a $c$-cyclically monotone set $\Gamma$, i.e.\ if $(x_1, y_1), \dots, (x_n, y_n)$ are in $\Gamma$, then, with $y_{n+1}= y_n$,  
\be
\sum_{i=1}^n c(x_i, y_i) \leq \sum_{i=1}^n c(x_i, y_{i+1}).
\ee 
There is an  equivalent way of defining $c$-cyclical monotonicity by means of finitely supported measures: a set  $\Gamma$ is $c$-cyclically monotone iff each   measure $\alpha$, which is finite and  concentrated on finitely many elements of $\Gamma$, is a cost minimizing transport between its marginals; that is, if  $\alpha'$ has the same marginals as $\alpha$,  then
$$\int c \; d\alpha \leq \int c \; d\alpha'.$$
The equivalence of the definitions follows easily from e.g.\     \cite[Thm. 6.1.4]{AmGiSa08}. % (cf.\ \cite[Exercise 2.21]{Vi03}) .
The second definition could be expressed also by using the set $\F_1$ and additionally requiring that $\alpha$ and $\alpha'$ have the same total mass.}
Our  general definition of $c$-monotonicity  applicable to \eqref{Primal} is the following:

\begin{definition}\label{DefinitionMon}
For a measure $\alpha$ on the Polish space $E$ and a set $\cF$ of measurable functions $E \rightarrow \RR$, a competitor of $\alpha$ is a measure $\alpha'$ on $E$ such that $\alpha (E) = \alpha'(E)$, and for all $f \in \cF$ one has 
\begin{align}
\int f \, d\alpha = \int f \, d\alpha'.
\end{align}
%If in addition
%\begin{align} 
%\textstyle \int c \, d \alpha' < \int c \, d\alpha, 
%\end{align}
%then $\alpha'$ is called a $c$-better competitor of $\alpha$.
If, in addition, $\alpha$ is finitely supported, i.e. concentrated on finitely many points, we require this property also for a competitor. \\
A set $\Gamma \subseteq E$ is called \emph{finitely minimal / $c$-monotone} if each  measure $\alpha$, which is finite and  concentrated on finitely many points in $\Gamma$, is cost minimizing amongst its  competitors.
A measure $\P$ is called \emph{finitely minimal/ $c$-monotone} if it is concentrated on a  finitely minimal / $c$-monotone set.
\end{definition}

%\section{Optimal martingale measures are finitely optimal }

Establishing that optimizers of problem \eqref{Primal} are finitely minimal will need an assumption on the family $\cF$: 
\begin{assumption}\label{BoundedAssumption}
\begin{enumerate}
\item There exists a function $g: E\to [0, \infty)$ such that each element of $\cF$ is bounded by some multiple of $g$. I.e.,\ for each $f\in \cF$ there is a constant $a_f \in \R_+$ such that $|f|\leq a_f g$. 
\item All functions in $\cF$ are continuous, or $\cF$ is at most countable.
\end{enumerate}
\end{assumption}
These properties are satisfied in all examples encountered in this article.

\begin{theorem}\label{MainTheorem}
Let $E$ be a  Polish space and $c: E \rightarrow (-\infty, +\infty]$ a Borel measurable function. Let $\cF$ be a family of Borel-measurable functions on $E$ satisfying Assumption \ref{BoundedAssumption} and assume that  $\P^*$ is such that
\beo
\min_{\P \in \cM_{\cF} } \int c \; d\P = \int c \; d\gamma^* \in \RR.
\eeo
Then $\P^*$ is finitely minimal / $c$-monotone. 
\end{theorem}
%In applications one usually works with continuous or lower semi-continuous cost functions. In that  case the existence of an optimizer $\P^*$ can often be established by compactness arguments. However, (semi-)continuity does not simplify our arguments nor does it lead to a more specific result. For instance, for the Monge-Kantorovich problem,  one obtains a nicer result for the most relevant case in which $c$ is continuous: the support of an optimal transport plan  is $c$-cyclically monotone.  But this assertion need not be true in our setup. Juillet \cite{Ju14}  gives an example of a two-period  martingale transport problem in which the marginals $\mu, \nu$ are compactly supported,  the cost function $c(x,y)= (y-x)^3$ is  continuous,  the minimizer is unique and its support is not finitely optimal.
%We have therefore chosen to go with the general formulation above. 

\subsection{Connection with \cite{Za15}}\label{ZaevSection}
In independent work,  Zaev \cite{Za15}  obtains (among a number of further developments) a result which is related to Theorem \ref{MainTheorem}. %; \cite{Za15} works in the setup of a multi-marginal transport problem allowing for additional linear constraints. We will discuss the precise relation in Section \ref{ZaevSection} below. 
 His article  is concerned with the multi-marginal transport problem described in Section \ref{MuMa},  allowing for additional linear constraints. In our notation this corresponds to problem \eqref{Primal} on a set $E$ which is a product $X_1 \times \ldots \times X_n$ of Polish probability spaces and where $\F$ is a superset of the set $\F_2$ defined in \eqref{MultF}; several important extensions of the transport problem can be phrased in this form. Under continuity and (weak) integrability assumptions Zaev establishes the existence of an optimizer, an extension of the classical Monge-Kantorovich duality as well as a necessary geometric condition for optimizers. The latter statement is equivalent to the assertion of Theorem \ref{MainTheorem} (applied to the setup of \cite{Za15}). The proof given in \cite{Za15} is based on his duality  result and different from the approach pursued here.

\section{Examples}

\subsection{Optimal transport and its multi-marginal version}\label{MuMa}

The Monge-Kantorovich problem \eqref{G1} fits the framework of \eqref{Primal}:  a measure $\P$  on $E= X \times Y$ is a transport plan in $\Pi(\mu, \nu)$ if and only if  
$$ \int \phi (x)\, d\P (x,y)=\int \phi(x)\, d\mu(x),$$
and 
$$\int \psi (y)\, d\P (x,y)=\int \psi(y)\, d\nu(y)$$
for all continuous bounded functions $\phi: X\to \R, \psi:Y\to \R$. Therefore \eqref{G1} is equivalent to \eqref{Primal} with 
$$ \F_1=\left\{ \phi\circ p_X - \int \phi \, d\mu : \phi \in C_b(X) \right\}
\cup
\left\{\psi \circ p_Y-\int \psi\, d\nu:  \psi \in C_b(Y)\right\}.$$ 
Regarding its statement, the multi-marginal problem is mainly an extension in notation: $\mu_1, \ldots, \mu_n$ are probability measures on Polish spaces $X_1, \ldots, X_n$, the  set $\Pi(\mu_1, \ldots, \mu_n)$
consists  of the probability measures $\P$ on $E=X_1\times \ldots \times X_n$ with  $p_i(\P)= \mu_i$ for $ i= 1,\ldots, n$,  and the problem is to find
\begin{equation}\label{MultG1}
\inf\left\{ \int c\, d\P:\P\in \Pi(\mu_1,\ldots, \mu_n)\right\},\,
\end{equation}
which is equivalent to \eqref{Primal} with 
\begin{align}\label{MultF}
\F_2= \left\{  \phi \circ p_i-\int \phi\, d\mu_i: \phi\in C_b(X_i), 1\leq i \leq n\right\}.
\end{align}

Note that $c$-monotonicity for these problems is just cyclical monotonicity as it is stated in the introduction.

\subsection{Model-independent finance -- Martingale Transport}\label{MIMT} 

Starting with the Monge-Kantorovich problem, but looking for an optimizer only among martingale measures, yields the problem of optimal martingale transport (in its most basic formulation). It is closely related to model independent finance, a field that is concerned with determining the possible price range of financial assets under the martingale-paradigm of mathematical finance, see for instance \cite{Ho11, AcBePeSc16, HeTo13, BoNu15}.
Roughly speaking, the payoff of a financial asset is represented by a cost function depending on the evolution of the price of an underlying stock. Due to the martingale-pricing paradigm, an arbitrage-free price of the asset is computed as its expected payoff under a martingale measure that is calibrated to market information. The task is hence to find minimum and maximum prices with the help of suitable martingale measures.

Here we have $E=\R_+^n$ or $\R^n$, and  $c: E \to \R$. A probability measure $\P$ on $E$ is a martingale measure %if the   coordinate process on $\RR^n$ is a martingale (in its own filtration) with respect to $\P$.   
iff for each $l < n$ one has equality and real values in 
\beo 
\int x_{l+1} \; \varphi (x_{1}, \dots, x_{l})    \; d\P = \int x_{l} \; \varphi (x_{1}, \dots, x_{l})  \; d\P
\eeo
for each continuous bounded   function $\varphi: \RR^l \rightarrow \RR$.
If we can observe the  current value $\xi \in \RR$ of the stock price,  we  only have to  consider martingales where all marginals have expectation $\xi$.

We therefore consider 
\begin{align}
\F^{(mart)} =  \{ p_{1} - \xi \} \cup  \left\{  (p_{l+1}-p_{l}) \, \left( \varphi \circ  p_{\{1, \dots, l \}}\right)  :   \varphi \in C_b\left( \RR^l \right) ,1\leq l < n \right\}.   
\end{align}

%For a general overview we refer to the survey of Hobson \cite{Ho11}.  
%Recent contributions on the general theory in discrete time include \cite{AcBePeSc13, HeTo13, BoNu13}.
%Here $E=\R_+^n$ or $\R^n$, and any $n$-tuple $(x_1, \ldots, x_n)$ is interpreted as a possible evolution of the stock price at future dates $t_1< t_2 < \ldots < t_n$. A possible price of a ``path-dependent option'' with payoff is then calculated as an integral
%\begin{align}\label{OptionInt} \int c \, d\P.\end{align} 
%Model-independent finance is
%Êabout determining the minimal (or maximal) possible prices subject to appropriate  constraints, i.e.\ about optimizing \eqref{OptionInt} over a suitable class of probabilities $\P$.

%According to the martingale pricing paradigm in mathematical finance 
%the measures of interest are \emph{martingale measures}, 
%This leads us to consider the family of functions
%\begin{align}
%\F^{(m)} =  \{ p_{1} - \xi \} \cup  \left\{  (p_{l+1}-p_{l}) \, \left( \varphi \circ  p_{\{1, \dots, l \}}\right)  :   \varphi \in C_b\left( \RR^l \right) ,1\leq l < n \right\}.   
%\end{align}

The martingale condition (with expectation $\xi$) then corresponds to  $\int f\, d\P=0$ for all $f\in \F^{(mart)}$. 
%There usually is additional information derived from market-data,  again corresponding to  $\int f \, d\P=0$ for $f$ in some family of functions $\H$. 
%For $\H$ we  list some choices of particular interest:  $\H= \emptyset$ is not relevant for mathematical finance but more so in probability through its connection to martingale inequalities: we refer to \cite{AcBePeScTe13, BeSi13, BoNu13, BeNu14} for recent developments in this direction. A noteworthy result of Bouchard and Nutz \cite{BoNu13} is that \emph{every} martingale inequality in finite discrete time can be derived from a ``dual'', elementary and deterministic inequality. 

Further market information can be encoded through additional functions. 
For instance, it is often  a reasonable idealization to assume that the marginal distributions of the stock price at  particular time instances can be derived from market data. The case of a given marginal distribution at the  terminal time $t_n$  has been particularly intriguing.\footnote{This case is naturally connected to the Skorokhod embedding problem, we refer to the survey of \OB\ \cite{Ob04}.} In the present context this corresponds to $p_n(\P)=\mu$ for some probability  $\mu$, i.e.\ specifying 
\begin{align}\label{MultF}
\H= \left\{  \phi \circ p_n-\int \phi\, d\mu: \phi\in C_b(\RR)\right\}.
\end{align}
More recently also the case with  all intermediate marginals given  has been considered. This corresponds to $\H=\F_2$ (where $X_1= \ldots= X_n=\R$).
The principal  problem of model independent finance can hence be seen as \eqref{Primal} with $\cF_4=\F^{(mart)}\cup\H$. 

The article \cite{BeJu16} discussed the case with $\F = \F^{(mart)} \cup \F_2$ and $n=2$. The notion of \emph{finite optimality}  introduced there can easily be seen to be equivalent with $c$-monotonicity for this problem. In fact, that notion and the variational lemma  in \cite{BeJu16} characterizing optimality via finite optimality have served as a basis for Definition \ref{DefinitionMon} and Theorem \ref{MainTheorem}.

\secret{
Note that for finitely supported finite measures, the conditions (1) and (2) above are equivalent to (1) and (2'):
{
\renewcommand{\labelitemi}{(2')}
\begin{itemize}
\item for each bounded Borel-measurable $\varphi$, we have 
\beo
\int  (x_{2} - x_{1})\; f (x_{1})   \, d \alpha (x_{1}, x_{2}) = 
%\int\int  (x_{1}) \; (x_{2} - x_{1}) d \alpha_{x_{1}}(x_{2}) \alpha^{\pi_{1}} (x_{1}) =
%\eeo
%\be
%= \int \int \rho (x_{1}) \; (x_{2} - x_{1}) d\alpha_{x_{1}}'(x_{2}) d\alpha'\;^{\pi_{1}}(x_{1}) = 
\int  (x_{2}-x_{1}) \; f (x_{1})  \, d \alpha' (x_{1}, x_{2}).
\eeo
\end{itemize}
}
Now (1) and (2') can be written in short as
\beo
\int  f \; d \alpha = \int f \; d \alpha',  ~~ \forall f\in \F^{(m)}\cup \F_2.
\eeo
%for all $f\in \F^{(m)}\cup \F_2$.
}

\subsection{A counterexample to sufficiency}
It is natural to ask whether the converse of Theorem \ref{MainTheorem} holds true as well, i.e.\ if finite optimality is also sufficient for optimality overall, at least under additional regularity assumptions on the function $c$ and the underlying spaces. This is not the case as shown by the following counterexample in the context of transport plans which are invariant under group actions (see e.g.\ \cite{KoZa13}). %meaningless in the sense that every set is finitely optimal. For instance, our criterion will not be able to deliver insights into martingale transport problems in continuous time. 
%A more interesting counterexample  considers  another variant of the transport problem $\eqref{G1}$. It shows in particular that finite optimality cannot be sufficient for optimality in the general framework above, although it is known that it is sufficient under mild regularity assumptions in optimal transport, and under less mild assumptions in martingale transport (see \cite{BeJu16}):
\begin{example}  
Let $X=Y=(0,1)$, and $\mu = \nu = \lambda$. For some irrational number $ \xi > 0 $,  let $T: (0,1) \rightarrow (0,1)$ denote the operator $x \mapsto x \oplus \xi$ (addition of $\xi$ modulo 1). We want to minimize the cost $c(x,y) = (y-x)^2$ among the transport plans $\pi$ that are $T \otimes T$-invariant, i.e.\ the transport plans $\pi$ for which $\pi = \bigl( T\otimes T \bigr) (\pi)$. These transport plans are characterized as those for which 
\beo
\int h \bigl( T\otimes T \bigr) \, d\pi = \int h \, d \pi ~\text{ for all } h\in C_{b}(X\times Y).
\eeo
The unique minimizer here is the uniform distribution on the diagonal, but each other transport plan is also concentrated on a finitely minimal set, as each subset of $X\times Y$ is finitely minimal: every finite and finitely supported $\alpha$ is its only competitor. For a competitor $\alpha'$, the signed measure $\alpha - \alpha'$ is  $T\otimes T$-invariant, and hence   a continuous measure. The only finitely supported such measure is zero, hence $\alpha = \alpha'$.
\end{example}

\section{A continuum marginal transport problem revisited}

In this section we discuss in some detail the problem  introduced by Pass  in \cite{Pa13}. For an interval $I = [0, T]$ we consider  a family of probability measures on $\RR$, $(\mu_t)_{t\in I}$ such that $t\mapsto\mu_t$ is weakly continuous. 
We consider the  space $\cR [0,T]$ of  Riemann-integrable functions
 $[0,T]\rightarrow \RR$ and 
write $\Pi_{\cR}( \mu_{t} )$ for the set of probability measures with marginals $(\mu_{t})_{t\in I}$ on the space $\cR [0,T]$.  ($\Pi_{\cR}( \mu_{t} )$ is non-empty, see Lemma \ref{NonEmptyR} below.)

Given a concave function $h: \RR \mapsto \RR$, the goal is to determine
\begin{align}\label{Pass} \tag{B}
\inf_{\P \in \Pi_{\cR}(\mu_{t})}  \int h \Bigl( \int_{0}^T f(t) \, dt \Bigr) \, d\P(f).
\end{align}
%The family $(\mu_{t})_{t\in I}$ can be assumed to be weakly continuous, as otherwise $\Pi_{C}( \mu_{t} )$ is empty.
$\cR [0,T]$ is not a Polish space, and so this problem is not exactly  an instance of \eqref{Primal}. We will nevertheless be able to use Theorem \ref{MainTheorem} for deriving an optimality result.
% this time with  $E=C[0,T]$ and 
% \begin{align}\label{PassF}
% \F_3= \left\{  \phi \circ p_t-\int \phi\, d\mu_t: \phi\in C_b(\R), t\in [0,T]\right\}.
% \end{align}
%Pass establishes a seemingly natural solution, but with a proof building on several technical assumptions.

We denote by $q_{t}: (0,1) \rightarrow \RR$ the quantile function 
%\footnote{I.e.\ $q_t$ is the generalized inverse of the cumulative distribution function of $\mu_t$: $q_t(x)= \inf\{ y: \mu_{t}\bigl( (-\infty, y]\bigr) \geq x \}$} 
of $\mu_{t}$, i.e.\ $q_t$ is the generalized inverse of $\mu_t$'s distribution function:
$q_t(x)= \inf\{ y: \mu_{t}\bigl( (-\infty, y]\bigr) \geq x \}$. 
The map $q: (0,1) \rightarrow \RR^{[0,T]}$, $x \mapsto q_{.}(x)$ pushes forward Lebesgue measure $\lambda$ from $(0,1)$ to a measure $\pi^*$ on $\RR^{[0,T]}$ that can be described as a uniform distribution on the quantile paths of $(\mu_{t})$. 

Notably $t\mapsto q_t(x)$ is in general not continuous\footnote{Continuity of $t\mapsto q_t(x)$ is claimed to follow from continuity of $t\mapsto \mu_t$ in \cite{Pa13}, but this is a glitch.}, consider e.g.\ $T=1$ and $\mu_t=t\delta_{\{0\}} + (1-t)\delta_{\{1\}}$ (see Example \ref{NewCont}  for a counterexample  with a family of absolutely continuous measures).
In fact,  there might even be quantile paths that are not Riemann-integrable, we present the (somewhat lengthy) argument in  Example \ref{NotRINT} at the end of this section.

Nevertheless, the measure $\pi^*$ can be regarded as a measure on $\cR[0,T]$ due to 

\begin{lemma}\label{NonEmptyR}
If $(\mu_{t})_{t\in [0,T]}$ is weakly continuous, then  for the measure $\pi^*$ described above, $\pi^*\bigl(\cR[0,T] \bigr)= 1$.
\end{lemma}
%\begin{remark}
%It is claimed in \cite{Pa13} that (in the setup of Lemma \ref{NonEmptyR}) $\pi^*\bigl(C[0,T] \bigr)= 1$, where $C[0,T]$ denotes the space of continuous functions. Correspondingly Pass \cite{Pa13} works with continuous functions rather than Riemann-integrable functions throughout. However, this assertion is in general not true, e.g.\ for $T=1$, $\mu_t =t \delta_{\{0\}}+(1-t)\delta_{\{1\}}$ we have $\pi^*\bigl(C[0,T] \bigr)=0$.
%\end{remark}
\begin{proof}%[Proof of Lemma \ref{NonEmptyR}.]
First note that each quantile path is bounded on the compact interval $[0,T]$. This is an easy consequence of weak continuity.\footnote{ Otherwise there would be a convergent sequence $t_n \rightarrow t$ such that either $q_{t_n}(x) \rightarrow \infty$ or $q_{t_n}(x) \rightarrow -\infty$. We only consider the first case: pick  a $y> x$ and set $G= q_t(y)$. There is an $N$ such that for all $n\geq N$ we have $q_{t_n}(x) > 2G$, and hence also $q_{t_n}(x') > 2G$ for all $x'>x$. We can find an $x' \in (x,y)$ such that $q_t(.)$ is continuous in $x'$. Therefore, $q_{t_n}(x')$ should converge to $q_t(x')$ but this is impossible as 
$q_t(x') \leq G < 2G < q_{t_n}(x')$.
}
 We hence need to show that for $\lambda$-a.e. $x\in (0,1)$ the path $t\mapsto q_t(x)$ is continuous in $\lambda$-a.e. $t\in [0,1]$.
Set 
\beo 
U = \{ (x,t)\in (0,1)\times [0,1]: q_{.}(x) \text{ not continuous in } t \}
\eeo
and 
\beo
U' = \{ (x,t)\in (0,1)\times [0,1]: q_{t}(.) \text{ not continuous in } x \}.
\eeo
We have $U \subseteq U'$ due to weak continuity. 
$U'$ is a Borel set: note that $q_{t}(.)$ is continuous in $x$ if and only if it is right-continuous in $x$, and the function $(x,t) \mapsto q_t(x)$ is measurable. So $U'$ is the complement of the  Borel set
\beo
\bigcap_n \bigcup_m \; \Bigl\{ (x,t): q_t(x+\dfrac{1}{m}) - q_t(x) < \dfrac{1}{n} \Bigr\}.
\eeo
Moreover, $U'$ is a null set by Fubini's theorem. Therefore, $U$ is a Lebesgue null set, and we must have 
$\lambda (U_x)$ for $\lambda$-a.e. $x$.
\end{proof}

  Pass shows that   $\pi^*$ is the unique minimizer of \eqref{Pass}. Among other conditions, he is building on the assumption  that the quantile functions satisfy a property of uniform Riemann-integrability which may be difficult to verify in practice.  We will show the following strengthened result:
 %and gives several surprising  applications from parabolic equations to mathematical finance and quantum physics. 
%On the technical side, however, his proof assumes that the quantile functions satisfy a property of uniform Riemann-integrability which may be difficult to verify in practice.  
%In section 4, we will discuss this problem using Theorem \ref{MainTheorem} and establish the following strengthened version:

\begin{theorem}\label{Pass-Theorem}
Let $h:\RR \rightarrow \RR$ be concave and $( \mu_{t})_{t\in I}$ a family of probability measures on $\RR$, weakly continuous in $t$ and such that $$\int_0^T \int |x|\, d\mu_t(x) \, dt < \infty, ~ \text{ and } \int |h|  \, d\mu_t  < \infty  ~\text{ for all } t\in [0,T].$$ Then $\pi^{*}$ is a minimizer of  \eqref{Pass}. 
%If  $h \in L_{1}(\mu_{t})$ for all $t\in I$, then $\pi^{*}$ is a minimizer of  \eqref{Pass}. 
If the infimum in \eqref{Pass}  is finite and $h$ is strictly concave,  then $\pi^*$ is the unique minimizer.
\end{theorem}

W.l.o.g. we work with  $I=[0,T]=[0,1]$ from now on.

For completeness and to fix ideas, we discuss a result that can be seen as a finite-dimensional predecessor to \cite{Pa13} and has been well-known for at least several decades. We mention the note by  \cite{Kaa02} for a simple geometric proof and further references, and for a   more general result   \cite{Car03}. We denote by $\pi^*_{n}$ the $n$-dimensional analogue of the measure $\pi^{*}$ introduced above. I.e.,  given $n$ probability measures $\mu_{{1}}, \dots, \mu_{{n}}$ on $\RR$, $\pi^*_{n}$ is the push forward of Lebesgue measure $\lambda$ on $(0,1)$ to $\RR^n$ via  $x \mapsto \bigl( q_{1}(x), \dots, q_{n}(x) \bigr)$, where, as before, $q_{i}$ is the quantile function of $\mu_{i}$.

\begin{theorem}\label{Folk-Theorem}
Let $h: \RR \rightarrow \RR$  be strictly concave and $\mu_{1}, \dots, \mu_{n}$ be probability measures on $\RR$ such that 
$$
\int  |x| \, d \mu_{i} < \infty, ~~ \int |h| \, d \mu_i < \infty, ~\text{ for }1 \leq i \leq n.
$$
Then $\pi^{*}_{n}$  is the unique minimizer of
\be \label{Folk-Problem}
\inf_{\P \in \Pi_{n}(\mu_{{1}}, \dots, \mu_{{n}})} \int h(x_{1}+\dots+x_{n}) \, d\P(x).
\ee
\end{theorem}

%\begin{theorem}\label{Pass-Theorem}
%Let $h:\RR \rightarrow \RR$ be strictly concave, and $\bigl( \mu_{t}\bigr)_{t\in I}$ a family of probability measures on $\RR$, weakly continuous in $t$. 
%If  $h \in L_{1}(\mu_{t})$ for all $t\in I$, then $\pi^{*}$ is a minimizer of  \eqref{Pass}  If the minimum is finite, then $\pi^*$ is the unique minimizer.
%\end{theorem}

It is intuitive to see why the monotonicity principle comes in useful for Theorem \ref{Folk-Theorem}. Finite  optimality  of a set $A$  (in $\RR^n$)  implies that $A$ must be a monotone set, i.e.\ $\leq$ must be a total order on $A$: if $f$ and $g$ are both in  $A$, then either $f \leq g$ or $g \leq f$. Else,  set  $f' = \max\{ f, g \}$ and  $g' =  \min \{ f, g \}$, and let $\alpha$ be the measure $\frac{1}{2} \delta_{f} + \frac{1}{2} \delta_{g}$ and $\alpha'$ the measure $ \frac{1}{2}\delta_{f'}+ \frac{1}{2}\delta_{g'}$. Then $\alpha'$ is a measure with the same marginals as $\alpha$ (on $\RR^n$, or $\cR [0,1]$, respectively).  But due to strict concavity of $h$, it is easy to see that $\alpha'$ leads to lower costs than $\alpha$ in  both cases \eqref{Folk-Problem} and \eqref{Pass}, contradicting the definition of finite optimality.    The argument of optimality of $\pi^{*}_{n}$ (or $\pi^{*}$, respectively) is then completed by another well-known fact, a proof of which we include for the convenience of the reader: %, cf. \cite[Lemma 1.4]{Ju14}. % A proof of the first part can be found e.g.\ in \cite[p 108]{RaRu98} or \cite[134]{AmGiSa08}. 
\begin{lemma}\label{Uniqueness-Lemma}
Let  $\P$ be a probability measure on $\RR^n$ with  marginals $\mu_{1}, \dots, \mu_{n}$. If there is a monotone Borel set $M$ with $\gamma(M)=1$, then $\P = \pi^{*}_{n}$.\\
Let $\P$ be a probability measure on $\RR^{[0,1]}$ with marginals $\bigl(\mu_{t}\bigr)_{t\in I}$. If there is a monotone Borel set $M$ with $\P(M)=1$, then $\P=\pi^{*}$.
\end{lemma}
\begin{proof}
The second part is a simple consequence of the first one since the distribution of a continuous time process is determined by its finite dimensional marginal distributions.  Hence, let $\P$ be as in the first statement. For arbitrary points $a_{1}, \dots, a_{n} \in \RR$, we   show that for $I = (-\infty, a_{1}] \times \dots \times (-\infty, a_{n}]$ we have 
$ \P (I) = \pi^{*}_{n}(I)$. Set $ z = \sup \{ x: q_{i}(x)  \leq a_{i} ~~\text{ for  } i = 1, \dots, n  \}$. 
Then we have $\pi^{*}_{n}(I) = z$, and for at least one $i_{0}$ we have $\mu_{i_{0}} \bigl( (-\infty, a_{i_{0}}] \bigr) = z$.  We can hence conclude that $\gamma (I) \leq z$. And, in fact, equality must hold. For observe that from the definition of $z$ we have  $\mu_{i} \bigl( (-\infty, a_{i}] \bigr) \geq z$ for all $i= 1, \dots, n$. Hence $\gamma (I) < z$ would imply that for each $i$ there is an element $(b^{(i)}_{1}, \dots, b^{(i)}_{n})  \in \Gamma$ such that $b^{(i)}_{i}Ê\leq a_{i}$, and  $b^{(i)}_{j_{i}} > a_{j_{i}}$ for some $j_{i} \neq i$. This contradicts the monotonicity of $M$.
\end{proof}

\begin{proof}[Proof of Theorem \ref{Folk-Theorem}]
The set  $\Pi (\mu_{1} \dots, \mu_{n})$ is weakly compact. Due to the assumptions on first moments and $h$-moments of the marginal measures $\mu_{i}$,  the operator to be minimized is lower semi-continuous and bounded. Hence there is a finite minimizer. 
Strict concavity of $h$ and the above outlined application of  the monotonicity principle yield that each finite minimizer  must be concentrated on a finitely minimal, hence monotone set. By the preceding  lemma, each minimizer must be equal to $\pi^{*}_{n}$. 
\end{proof}

Now we turn to proving  Theorem \ref{Pass-Theorem}:  %here the  neat argument for Theorem \ref{Folk-Theorem} does not work as smoothly as before,   as  $\Pi_{C}\bigl(\mu_{t}\bigr)$ need not be compact. This can be seen by easy counterexamples. Hence we have to find a way to establish the existence of an optimizer at all. Here is how we want to proceed: 
we will solve a problem for a countable index set as an intermediate step, where we also add monotonicity and boundedness (from above) to the assumptions on $h$. We then use the intermediate result in  the proof of Theorem \ref{Pass-Theorem} at the end of this section. Writing $Q =  [0,1] \cap \QQ$, we define $\Pi_{Q}(\mu_{q})$ as  the set of probability measures on $\RR^{Q}$ with marginals $(\mu_{q})_{q\in Q}$. Furthermore, we fix a sequence of finite partitions  $( \cP_{n} ) $ of $[0,1]$ with  $\cP_{n}\subseteq \cP_{n+1} \subseteq Q$ and  $\bigcup_{n} \cP_{n} = Q$. We then replace the  original problem  \eqref{Pass}   by 
\begin{align}\label{PassAux}\tag{B'}
\inf_{\P \in \Pi_{Q} (\mu_{q})} \int h \Bigl( \limsup_{n\rightarrow \infty} \sum_{t_{i}\in \cP_{n}} f_{t_{i}} (t_{i}-t_{i-1}) \Bigr) \, d\P(f).
\end{align}

Writing $\pi^{*}_{Q}$ for the $Q$-analogue of $\pi^*$, we claim: 

\begin{pro}\label{Pass-Q-Theorem}
Let $h: \RR \rightarrow \RR_{\leq0}$ be  non-positive, concave and increasing. Provided that $\int |x| \, d\mu_{q}(x) < \infty, ~\int |h| \, d \mu_{q} < \infty$ for all $q\in Q$,  the measure $\pi^{*}_{Q}$ is a minimizer of Problem {\eqref{PassAux}}.
\end{pro}

The proof is preceded by Lemmas \ref{Compactness Lemma}, \ref{Continuity Lemma}, and \ref{Optimality Lemma}. The assumptions here on $h$ and the marginals are  as in Proposition  \ref{Pass-Q-Theorem}.
\begin{lemma}\label{Compactness Lemma}
$\Pi_{Q} \bigl(\mu_{q} \bigr)$ is weakly compact.
\end{lemma}
\begin{proof} By Prokhorov's theorem: let $\varepsilon > 0 $ be arbitrary. Then, with $Q = \{ q_{1},  q_{2}, \dots \}$,  for each $q_{k}$ there exists a compact set $K_{k}\subseteq \RR$ with $\mu_{q_{k}}(K_{k}) > 1- \frac{\varepsilon}{2^k}$. The set $K = \Pi_{k=1}^{\infty} K_{k}$ is a compact subset of $\RR^{Q}$. For a measure $\P \in \Pi_{Q} \bigl(\mu_{q} \bigr)$ we have 
$$
\P (K) = \lim_{n \rightarrow \infty } \P \bigl( p_{q_{1}, q_{2}, \dots, q_{n}}^{-1} (K_{1} \times K_{2} \times \dots \times K_{n} ) \bigr).
$$
As for each $n$
$$
\P \bigl( p_{q_{1}, q_{2}, \dots, q_{n}}^{-1} (K_{1} \times K_{2} \times \dots \times K_{n} ) \bigr) > 1 - \sum_{k=1}^n \frac{\varepsilon}{2^k} \geq 1- \varepsilon
$$
we have $\P (K) \geq 1-\varepsilon$, and Prokhorov's theorem can be applied. \end{proof}
We introduce some notation: 

\begin{align*}
s_{n}\ &: \RR^{Q}  \rightarrow \RR, \ & & f \mapsto \sum_{ t_{i} \in \cP_{n} } f_{t_{i}} (t_{i}-t_{i-1}),\\
s_{n}^{(h)} \ &: \RR^{Q}   \rightarrow \RR, ~~~ \ & & f \mapsto \sum_{t_{i}\in \cP_{n}} h(f_{t_{i}}) (t_{i}-t_{i-1}),\\ 
\varphi_{n}\ &: \RR^{ Q }  \rightarrow \RR \cup \{\infty \}, ~~~ \ & & f \mapsto \sup_{k\geq n} s_{k}(f), \\
\varphi\ &: \RR^{Q} \rightarrow \RR \cup \{ -\infty, \infty \}, ~~~\ & & f \mapsto \inf_{n} \varphi_{n}(f) = \limsup_{n} s_{n}(f).
\end{align*}

We continue with 
\begin{lemma}\label{Continuity Lemma}
For each $n$, the operators defined on $\Pi_{Q}(\mu_{q})$,
\beo
S_{n}: \P \mapsto \int h \circ s_{n} \, d\P
\eeo
and 
\beo
\Phi_{n}: \P \mapsto \int h \circ \varphi_{n} \, d\P
\eeo are lower-semi-continuous (w.r.t. weak convergence) and have minimizers. The values of the minima are finite.
\end{lemma} 
\begin{proof}
The existence of minimizers will follow from lower-semi-continuity of the operators and compactness of $\Pi_{Q}(\mu_{q})$. Hence, let $(\P_{l})_{l\in \NN}$ be a sequence in $\Pi_{Q}(\mu_{q})$ converging weakly to some $\P_{0}$. \\
We have
\beo
\varphi_{n} \geq s_{n}
\eeo
and hence, by monotonicity and concavity of $h$ that 
\beo
h \circ \varphi_{n} \geq h \circ s_{n} \geq s_{n}^{(h)}.
\eeo
For each  $\P\in \Pi_{Q}\bigl(\mu_{q}\bigr)$, 
\beo
\int s_{n}^{(h)} \, d\P = \sum_{t_{i}\in \cP_{n}} (t_{i}-t_{i-1}) \int h (f_{t_{i}}) \, d\P(f) = \sum_{t_{i}\in \cP_{n}} (t_{i}-t_{i-1}) \int h \, d\mu_{t_{i}}.
\eeo
Hence in particular 
\beo
\lim_{l \rightarrow \infty} \int s_{n}^{(h)} \, d\P_{l} = \int s_{n}^{(h)} \, d\P_{0}.
\eeo
As $s_{n}^{(h)}$ is continuous,  the prerequisites of Lemma 4.3 in \cite{Vi09} are met for both $S_{n}$ and $\Phi_{n}$, and applying that result we get 
\beo
\liminf_{l\rightarrow \infty} S_{n}(\P_{l}) \geq S_{n}(\P_{0})
\eeo
and
\beo
\liminf_{l\rightarrow \infty} \Phi_{n}(\P_{l}) \geq \Phi_{n}(\P_{0}).
\eeo
Finally, the finiteness of the minimal values follows  from $h$ being bounded from above, the assumption on finite $h$-moments of the marginals, and \\ $h \circ \varphi_{n} \geq h \circ s_{n} \geq s_{n}^{(h)}$.
 \end{proof}

\begin{lemma}\label{Optimality Lemma}
For each $n\in \NN$, the measure $\pi^{*}_{Q}$  minimizes $\Phi_{n}$ on $\Pi_{Q}\bigl(\mu_{q}\bigr)$. 
\end{lemma}
\begin{proof}
We  first show that, when $h$ is strictly concave, the following stronger assertion is true:  $\pi^{*}_{Q}$ is the unique measure in $ \Pi_{Q}\bigl( \mu_{q} \bigr) $ doing the following:

\medskip

\noindent \hspace{2mm}(0) it minimizes $\Phi_{n}$, 

\noindent\hspace{2mm}(1) among the minimizers of $\Phi_{n}$ it minimizes $S_{1}$,

\noindent\hspace{2mm}(2) among the measures fulfilling (0) and (1), it minimizes $S_{2}$,

$\vdots$ 

\noindent\hspace{2mm}$(k)$ among the measures fulfilling (0), (1), $\dots, (k-1)$, it minimizes $S_{k}$

$\vdots$  
\medskip

We  show existence of a measure fulfilling all the conditions $(0), (1), \dots$:
write $K_{0}$ for the set of minimizers of $\Phi_{n}$. By the previous lemma, $K_{0}  \neq \emptyset$. Also, $K_{0}$ is compact: for it is a closed subset of the compact set $\Pi_{Q}(\mu_{q})$, where closedness is due to the semi-continuity of $\Phi_{n}$. Hence, among the minimizers of $\Phi_{n}$, there is a minimizer of the lower-semi-continuous operator $S_{1}$. Writing $K_{1}$ for the set of  these minimizers, by the same argument as above, $K_{1}$ is nonempty and compact. Hence, the set $K_{2}$ of minimizers of $S_{2}$ on $K_{1}$ is nonempty and again compact. By induction we obtain a decreasing sequence of compact nonempty sets $K_{k}$. Hence the set $K=\bigcap_{k} K_{k}$ is nonempty and each of its elements fulfills properties $(0), (1), \dots$ Pick such  an element and denote it by $\pi_{0}$. We now apply the monotonicity principle to show that $\pi_{0}$ must indeed  be equal to $\pi^{*}_{Q}$:
$\pi_{0}$ is concentrated on a set $\Gamma$ that is finitely optimal for each of the problems $(k)$. Observe first that finite optimality of  $\Gamma$ for problem (0) alone  does not need to imply that $\Gamma$ is monotone.\footnote{What finite optimality does imply is the following: if $f,g$ are in $\Gamma$, and $\varphi_{n}(f) > \varphi_{n}(g)$, then one must have $\varphi_{n}\bigl( (f-g)^+\bigr) = 0$. This is a weaker condition than $\leq$ being an order on $\Gamma$, and explains  why one works with the sequence of problems $(k)$ rather than just with problem (0). } 
However, finite optimality of  $\Gamma$ for problem (1) - i.e.\ the optimization of $S_{1}$ on the set $K_{0}$ -  does imply that $\Gamma$ must be monotone on $\cP_{1}$, that is, if $f,g\in \Gamma$, then either $f|_{\cP_{1}} \leq g|_{\cP_{1}}$ or $f|_{\cP_{1}}\geq g|_{\cP_{1}}$. For if there were $f,g$ not ordered on $\cP_{1}$, then write $f' = \indy_{\cP_{1}} \max(f,g) + \indy_{\cP_{1}^c} f$ and $g'= \indy_{\cP_{1}}\min(f,g) + \indy_{\cP_{1}^c}g$. Set $\alpha = \frac{1}{2} \delta_{f} + \frac{1}{2} \delta_{g}$ and $\alpha' = \frac{1}{2} \delta_{f'} + \frac{1}{2} \delta_{g'}$, where $\delta_{f}$ denotes the Dirac-measure on $f$, etc. Then apparently $S_{1} (\alpha') < S_{1} (\alpha) $, but $\alpha'$  is also a competitor of $\alpha$: it clearly has the same marginals, and we have $\varphi_{n}(f') = \varphi_{n}(f)$ and $\varphi_{n}(g')= \varphi_{n}(g)$, as manipulating a function $f\in \RR^{Q}$ on finitely many points does not change the value of $\varphi_{n}$. Hence, also $\Phi_{n}(\alpha') = \int h\circ \varphi_{n}\, d\alpha' = \int h\circ \varphi_{n} \, d\alpha = \Phi_n (\alpha)$. The existence of an $S_{1}$-better competitor is a contradiction to finite optimality, so $\Gamma$ must indeed be monotone on $\cP_{1}$.
Now for problem (2), we also find  that $\Gamma$ must be  monotone on $\cP_{2}$: let $f,g \in \Gamma$, and assume, due to monotonicity of $\Gamma$ on $\cP_{1}$, that  $f|_{\cP_{1}}\geq g|_{\cP_{1}}$. If $f$ and $g$ were  not ordered on $\cP_{2}$, then the same construction of $f'$, $g'$,  $\alpha$ and $\alpha'$ as above (with $\cP_{2}$ in place of $\cP_{1}$) will give a contradiction to finite optimality: note that $s_{1}(f')= s_{1}(f)$ and $s_{1}(g')=s_{1}(g)$, as $f'=f$ and $g'=g$ on $\cP_{1}$. Hence, $\Phi_{n} (\alpha') = \int  h\circ \varphi_{n} \, d \alpha' = \int h\circ \varphi_{n} \, d\alpha = \Phi_{n}(\alpha)$, $S_{1}(\alpha') = \int h\circ s_{1} \, d\alpha' = \int h\circ s_{1}\, d\alpha = S_{1}(\alpha)$, and  $\alpha'$ is really a competitor of $\alpha$.\\
Iterating this argument one finds that $\Gamma$ must indeed be monotone on each $\cP_{k}$, and  henceforth monotone. 
But then $\pi_{0}$ must be $\pi^{*}_{Q}$, because $\pi^{*}_{Q}$ is the only measure in $\Pi_{Q}(\mu_{q})$ concentrated on a monotone set. This last statement follows easily from Lemma \ref{Uniqueness-Lemma}.

Finally, we discuss the case where $h$ is concave, but not necessarily strictly concave. Then, due to the finiteness of $\int |x| \, d\mu_{q}$ for all $q\in Q$, there is, for each $k\in \NN$, a strictly concave function $h_{k}$ such that $\int |h_{k}| \, d\mu_{q} < \infty $ for all $q \in \cP_{k}$. 
Then by adapting the above argument, it is easy to see that $\pi^{*}_{Q}$ is the only measure in $\Pi_{Q}\bigl( \mu_{q} \bigr)$ that

\medskip

\noindent \hspace{2mm}(0) minimizes $\Phi_{n}$

\noindent \hspace{2mm}(1') among the minimizers of $\Phi_{n}$, it minimizes $\int h_{1} (s_{1}) \, d\gamma$,

$\vdots$

\noindent \hspace{2mm}(k') among the measures fulfilling (0),  \ldots,  (k-1'), it minimizes $ \int h_{k} ( s_{k}) \, d\gamma$, 

$\vdots$

\end{proof}

\begin{proof}[Proof of Proposition \ref{Pass-Q-Theorem}]
Let $\P$ be a measure in $\Pi_{Q}(\mu_{q})$. Then for each $n$, according to the previous lemma
\beo
\int h\circ \varphi_{n} \, d\P \geq \int h \circ \varphi_{n} \, d\pi^{*}_{Q}.
\eeo

As $h$ is increasing and non-positive, and $\varphi_{n}$ decreases to $\varphi = \limsup_{n}s_{n}$, an application of monotone convergence finishes the proof. 
\end{proof}
Finally we can prove Theorem \ref{Pass-Theorem}.
\begin{proof}[Proof of Theorem \ref{Pass-Theorem}]
First, note that due to the regularity assumption  $\int_{0}^{1} \int |x| d\mu_{t} \, dt < \infty$, it is w.l.o.g to  assume that $h$ takes values in $\RR_{\leq0}$ only: otherwise, we would replace it by $h-ax-b$ for suitable values $a,b\in \RR$, and take care of the fact that the integral $\iint_0^1 f(t) dt \, d\gamma(f)$ is finite and invariant among all $\gamma \in \Pi_{\cR}( \mu_{t} ) $. This is a consequence of the said regularity assumption and Fubini's theorem applied to the function $(t,f) \mapsto f(t)$, which, on the product $[0,1] \times \cR$ is $\lambda \otimes \gamma$-a.e. equal to a measurable function. 
 
If we further assume for the time being that $h$ is increasing, we can apply Proposition \ref{Pass-Q-Theorem} to  see the optimality of $\pi^*$ as follows: 
let $p_{Q}$ be the projection $\RR^{I} \rightarrow \RR^{Q}$, and write $p$ for its restriction on $\cR [0,1]$. %It is easy to see that $p$ is a Borel isomorphism from $C[0,1]$ onto $\RR^{Q}_{c}$, the set of all elements of $\RR^{Q}$ that are restrictions of continuous functions on $[0,1]$.
For an arbitrary $\P \in \Pi_{\cR}(\mu_{t})$, the measure $p(\P)$ is in $\Pi_{Q}(\mu_{q})$ and clearly
\beo
\int h \Bigl( \int_{0}^1 f \, dt \Bigr) \, d\P = \int h \Bigl( \limsup_{n\rightarrow \infty } \sum_{t_{i}\in \cP_{n}} f_{t_{i}}(t_{i}-t_{i-1}) \Bigr) \, d \, p(\P).
\eeo
But for the right-hand-side one also has, due to Theorem \ref{Pass-Q-Theorem},
\beo
\int h \Bigl( \limsup_{n} \sum_{t_{i}\in \cP_{n}} f_{t_{i}}(t_{i}-t_{i-1}) \Bigr) \, d \, p(\P) \geq
\int h \Bigl( \limsup_{n} \sum_{t_{i}\in \cP_{n}} f_{t_{i}}(t_{i}-t_{i-1}) \Bigr) \, d \, \pi^{*}_{Q}.
\eeo
As the right-hand-side of this equals $\int h \Bigl( \int_{0}^1 f \, dt \Bigr) \, d\pi^{*}$ we have
\beo
\int h \Bigl( \int_{0}^1 f \, dt \Bigr) \, d\P \geq \int h \Bigl( \int_{0}^1 f \, dt \Bigr) \, d\pi^{*}.
\eeo

If $h$ is not increasing, then assume first it is decreasing. If in problem {\eqref{PassAux}} we replace  $\limsup$ by $\liminf$ one can show, with the   statement and proof of Proposition \ref{Pass-Q-Theorem} and the above argument suitably adapted, that $\pi^{*}$ must be again optimal.   Finally, if $h$ is neither increasing nor decreasing, then it can still be written as a sum $h_{1}+h_{2}$, where $h_{1}$ is  concave, increasing and non-positive, and $h_{2}$ is concave, decreasing and non-positive, and again $\pi^{*}$ is an optimizer. ($h_{1}$ and $h_{2}$ will satisfy the regularity assumptions as long as $h$ does.) 

If the minimum is finite and $h$ is strictly concave, each other minimizer must be concentrated on a finitely minimal, hence monotone set and thus be equal to  $\pi^{*}$.
\end{proof}

We close this section with two examples which complement  Lemma \ref{NonEmptyR} and may help to clarify the role of Riemann-integrability.

\begin{example}\label{NotRINT}
Weak continuity of $t\mapsto \mu_t$ does not imply that all quantile paths $x\mapsto q_t(x)$ are Riemann-integrable:

 First we construct a set $C'$ that bears some resemblance to the Cantor-set, but is considerably bigger:
starting with the unit interval $[0,1]$, we cut out the middle section $A_1 = \bigl(a_1^{(1)}, b_1^{(1)}\bigr)$ with length $1/10$. From the remaining two intervals, we  cut out the middle sections $\bigl(a_1^{(2)}, b_1^{(2)}\bigr), \bigl(a_2^{(2)}, b_2^{(2)}\bigr)$ such that their combined length is $1/100$. In the $n$-th step, we cut out the $2^{n-1}$ middle sections   $\bigl(a_1^{(n)}, b_1^{(n)}\bigr), \dots, \bigl(a_{2^{n-1}}^{(n)}, b_{2^{n-1}}^{(n)}\bigr)$ such that their combined length is $1/10^n$. We denote the whole set cut out in the $n$-th step by $A_n$ and set $C'= [0,1] - A_1 - A_2 - \dots$. We have $\lambda(C') = 8/9$, and its indicator function is continuous in $x$ iff $x\notin C'$. Therefore,  $\mathbbm{1}_{C'}$ is not Riemann-integrable.\\ 
Next, we construct a function $f$ on $[0,1]$ as the supremum of functions $f_n$: let $f_1$ be the function on $[0,1]$ that on $\bigl[ a_1^{(1)}, b_1^{(1)} \bigr]$ linearly interpolates between $(a_1^{(1)}, 0), ((a_1^{(1)}+ b_1^{(1)})/{2}, {1}/{10})$ and $(b_1^{(1)}, 0)$, and is equal to zero else. Let $f_n$ be the function that, for each $i\in \{1, \dots, 2^{n-1} \}$, on $\bigl[a_i^{(n)}, b_i^{(n)}\bigr]$ linearly interpolates between $(a_i^{(n)}, 0)$, $((a_i^{(n)}+b_i^{(n)})/{2}, {1}/{10^n})$ and $(b_i^{(n)},0)$ and is equal to zero else. Then set $f = \sup_n f_n$. Note that $f$ is continuous on $[0,1]$.\\
We consider the  family of probability distributions given by
\beo
\mu_t = \bigl({1}/{2}- f(t) \bigr) \delta_0 + \bigl( {1}/{2} + f(t) \bigr) \delta_1.
\eeo
Due to the continuity of $f$, the family $\bigl( \mu_t \bigr)_{t \in [0,1]}$ is weakly continuous in $t$.  For each $t$, the quantile function $q_t(.)$ jumps  from $0$ to $1$ after  $1/2 - f(t)$. We hence find that  
\beo 
q_t(1/2) = 1 - \mathbbm{1}_{C'}(t). 
\eeo
So the path $t \mapsto q_{t}(1/2)$ is not Riemann-integrable.
\end{example}

\begin{example}\label{NewCont} In Lemma \ref{NonEmptyR} it is not possible to replace the set of Riemann-integrable functions with the set of continuous functions, even if the family $(\mu_t)$ is  assumed to consist of absolutely continuous measures:

To see this, let $f$ be a function $[0,1] \rightarrow [0,1]$ with the following properties: $f$ is strictly increasing and absolutely continuous, $f(0)=0$, $f(1)=1$, $f(x) > x$ else. \\
Then define, for each $t \in [0,1]$,  distribution functions on $[0,1]$ by
\begin{eqnarray*}
F_t (x) = f(x) ~ \text{ for } x \in [0, t] \\
F_t (x) = f(t) ~ \text{ for } x \in [t, f(t)]\\
F_t (x) = x ~ \text{ for } x \geq f(t).
\end{eqnarray*}
It is obvious that the family $(\mu_t)_{t\in[0,1]}$ corresponding to the family $(F_t)_{t\in[0,1]}$ is weakly continuous in $t$, and all $\mu_t$ are absolutely continuous with respect to $\lambda$. 
But all quantile paths have a discontinuity: 
let $x \in (0,1)$, and $t$ such that $f(t) = x$. Then the quantile path of $x$, $q_{.}(x)= q_{.}(f(t))$, is discontinuous in $t$: 
\begin{equation*}
q_{t}(x) = q_t(f(t)) = t,
\end{equation*}
but for  all $t' < t$ we will have 
\begin{equation*}
q_{t'}(x) = q_{t'}(f(t)) = f(t).
\end{equation*}
But $f(t) > t$, so $q_{.}(x)$ is not continuous in $t$. 
\end{example}

\section{Proof of Theorem \ref{MainTheorem}}
In the proof of Theorem \ref{MainTheorem} we will make use of the following result from \cite{BeGoMaSc09},  which is a  consequence of a duality result by Kellerer \cite{Ke84}:
\begin{lemma}[{\cite[Proposition 2.1]{BeGoMaSc09}}]\label{KellLemma} Let $(E, m)$ be a Polish probability space, and $M$ an  analytic\footnote{\cite[Proposition 2.1]{BeGoMaSc09} is stated only for Borel sets, but the same proof applies in the case where $M$ is analytic. } subset of $E^l$, then one of the following holds true:\begin{enumerate}
\item[(i)] there exist $m$-null sets $M_{1}, \dots, M_{l} \subseteq E$ such that $M \subseteq \bigcup_{i=1}^l p_{i}^{-1} (M_{i})$, or
\item[(ii)] there is a measure $\eta $ on $E^l$ such that $\eta (M)>0$ and $p_{i}(\eta) \leq m$ for $i=1, \dots, l$.
\end{enumerate}
\end{lemma}
\begin{proof}[Proof of Theorem \ref{MainTheorem}.]  
We want to find a finitely minimal set $\Gamma$ with $\P^{*}(\Gamma) = 1$. To obtain this, it is sufficient to show that for each $l \in \NN$ there is a set $\Gamma_{l}$ with $\P^{*} (\Gamma_{l}) = 1$ such that:  for any finite measure $\alpha$ concentrated on at most $l$ points in $\Gamma_{l}$ and satisfying $\alpha(E) \leq 1$ as well as $\int g\, d\alpha \leq l, |c_{\upharpoonright\supp \alpha}| \leq l$, there  is no $c$-better competitor $\alpha'$ on at most $l$ points and satisfying  $\int g\, d\alpha' \leq l, |c_{\upharpoonright\supp \alpha'}| \leq l$.  For then  $\Gamma := \bigcap_{l\in \NN} \Gamma_l$ is finitely minimal.

Hence, fix $l$ and define a subset of $E^l$,
\begin{align*}
M &=  \{ (z_{1}, \dots, z_{l}) \in E^l :\\
 %\text{ for } \forall i, \, z_{j} \neq z_{k} \text{ for } \forall j \neq k :  
 &\textstyle 
 \exists \text{ a  measure }\alpha \text{ on } E, \alpha(E) \leq 1,  \int g\, d\alpha \leq l, \text{supp }\alpha\subset\{ z_{1}, \dots, z_{l}, |c_{\upharpoonright\supp \alpha}| \leq l \}, ~
\\  
&\textstyle\text{s.t.\ there is a $c$-better competitor }\alpha', \int g\, d\alpha' \leq l, |\supp \alpha' | \leq l, |c_{\upharpoonright\supp \alpha'}| \leq l \}. 
\end{align*}
Note that $M$ is the projection of the set 
\begin{align*}%\label{MotherSet}
\begin{split}
\hat M = & \Big\{ (z_1, \ldots, z_l, \alpha_1, \ldots, \alpha_l, z_1', \ldots, z_l', \alpha_1', \ldots, \alpha_l',) \in E^l \times \R_+^l \times E^l\times \R_+^l:\\
& \sum \alpha_i \leq 1, \sum \alpha_i g(z_i) \leq l, 
 \sum \alpha_i= \sum \alpha_i', \sum \alpha_i' g(z_i') \leq l, |c_{\upharpoonright\supp \alpha\cup \supp \alpha'}| \leq l \\
& \sum \alpha_i f(z_i)=\sum \alpha_i' f(z_i') \text{ for all $f\in \F$ },
 \sum \alpha_i c(z_i)>\sum \alpha_i' c(z_i')
\Big\}.
\end{split}
\end{align*}
onto the first $l$ coordinates.
By our Assumption \ref{BoundedAssumption}, the set $\hat M$ is Borel, hence $M$ is analytic.

We apply Lemma \ref{KellLemma} to  the space $(E, \P^*)$ and the set $M$: if (i) holds, then define $N := \bigcup_{i=1}^l M_{i}$. Then
$\Gamma_{l}:= E \setminus N$ has full measure,  $\P^{*}(\Gamma_{l})=1$.  From the definitions of $M$ and $N$ it can be directly seen that $\Gamma_{l}$ is as needed. 

If (i) does not hold, (ii) has to. Hence, let us derive a contradiction from it.

Write $p _{i}$ for the projection of an element of $E^l$ onto its $i$-th  component.
We may assume that the measure $\eta$ in (ii) is concentrated on $M$, and also  fulfills $p _{i}(\eta) \leq \frac{1}{l} \P^{*}$ for $i= 1, \dots, l$.

We now apply the Jankow -- von Neumann selection theorem to the set $\hat M$ to define a  mapping 
$$ z\mapsto \bigl( \alpha_1(z),\ldots, \alpha_l(z), z_1'(z), \ldots, z_l'(z),   
\alpha_1'(z),\ldots, \alpha_l'(z) \bigr) $$ 
such that 
$$ \bigl( z,\alpha_1(z),\ldots, \alpha_l(z), z_1'(z), \ldots, z_l'(z),   
\alpha_1'(z),\ldots, \alpha_l'(z) \bigr) \in \hat M $$
for $z\in M$, and the mapping is measurable with respect to the $\sigma$-field generated by the analytic subsets of $E^l$. 
Setting
$$
\alpha_z:=  \sum_{i} \alpha_i(z) \delta_{z_i}, 
\alpha'_z:=  \sum_{i} \alpha_i'(z) \delta_{z_i'(z)}
$$ 
we thus obtain kernels   $z\mapsto \alpha_z$, $z\mapsto \alpha'_z$ from $E^l$ with the $\sigma$-field  generated by its analytic subsets to $E$ with its Borel-sets. We use these kernels to define measures $\omega, \omega'$ on the Borel-sets on $E$ through 
\beo
\omega (B) =  \int \alpha_{z}(B) \, d\eta (z), \ \omega' (B) =  \int \alpha'_{z}(B) \, d\eta (z).
\eeo  
By construction $\omega \leq \P^{*}$. 
Moreover  $\omega' $ is  a $c$-better competitor of $\omega$:
for each $f\in \cF$ we have
\beo
\int f \; d \omega' = \int\!\!\!\!\int f \; d\alpha'_{z}d \eta (z) = \int\!\!\!\!\int f \; d\alpha_{z}d\eta (z) = \int f \; d\omega. 
\eeo
Note that the first and last equality are justified  since $\int g \, d\alpha_{z}$, $  \int g \, d\alpha'_{z} \leq l$ for all $z$.
Similarly, since $|c| \leq l$, $(\omega+\omega')$-a.s. we obtain
\beo
\int c \; d \omega' = \int\!\!\!\!\int c \; d\alpha'_{z}d \eta(z) < \int\!\!\!\!\int c \; d\alpha_{z}d\eta(z) = \int c \; d\omega. 
\eeo
%If we can construct a $c$-better competitor $\omega'$, then  the 
Summing up, we obtain a probability measure  $\P' := \P^{*} - \omega + \omega'$ with \\  $\int c \, d \P' < \int c \, d \P^{*}$ and $\P'\in \Pi_\F$. 
This contradicts the optimality of $\P^{*}$.
\end{proof}

\bibliographystyle{plain}

\bibliography{alle2}
%\bibliography{joint_biblio}
%\bibliography{BiblPaper1}
\end{document}